\documentclass[11pt]{article}
\usepackage{amsmath}
\usepackage{amssymb}
\usepackage{latexsym}
\usepackage{theorem}
\usepackage{mathrsfs}
\usepackage[all]{xy}
\usepackage{dave2}
\usepackage{stmaryrd}

\title{Classifying Dihedral $O(2)$-Equivariant Spectra}
\author{David Barnes}
\date{April 21, 2008}

\begin{document}
\maketitle
\begin{abstract}
\noindent 
The category of rational $O(2)$-equivariant spectra
splits as a product of cyclic and dihedral parts.
Using the classification of rational $G$-equivariant spectra
for finite groups $G$, we classify the dihedral part
of rational $O(2)$-equivariant spectra in terms of an algebraic model.
\end{abstract}

\section{Main Results}
The category of rational $O(2)$-spectra, $O(2) \mcal_\qq$,
is the category of $O(2)$-equivariant EKMM $S$-modules
(\cite{mm02}) with weak equivalences 
the rational $\pi_*$-isomorphisms. 
There is a strong
symmetric monoidal Quillen equivalence
$$
\Delta : O(2) \mcal_\qq 
\overrightarrow{\longleftarrow} 
L_{EW_+} O(2) \mcal_\qq \times L_{S^{\infty \delta}} O(2) \mcal_\qq
: \prod .
$$
Where $\delta$ is the determinant 
representation of $O(2)$ on $\rr$ and $W$ is the group 
of order two.
We call $L_{EW_+} O(2) \mcal_\qq$
the category of cyclic $O(2)$-spectra and 
$\dscr \mcal_\qq:=L_{S^{\infty \delta}} O(2) \mcal_\qq$
the category of dihedral spectra.
This splitting result was proven at 
the homotopy level in \cite{gre98a} and the model 
category statement above is \cite[Theorem 6.1.3]{barnes}.

An object $V$ of the category $\mathcal{A}(\dscr)$
consists of a differential graded rational vector space $V_{\infty}$
and $dg \qq W$-modules $V_k$ for each $k \geqslant 1$,
with a map of $dg \qq W$-modules
$\sigma_V \co V_\infty \to \colim_n \prod_{k \geqslant n} V_k$.
A map $f \co V \to V'$ in this category consists of 
a map $f_\infty \co V_{\infty} \to V_{\infty}'$
in $dg \qq \leftmod$ and 
$dg \qq W$-module maps 
$f_k \co V_k \to V_k'$
such that $\sigma_{V'} \circ f_\infty = \colim_n \prod_{k \geqslant n} f_k  \circ \sigma_V$.

To define $\mathcal{A}(\dscr)$, we have simply 
taken the algebraic model for the homotopy category
of dihedral spectra in \cite{gre98a} (a graded category)
and added the requirement that each piece have a differential.
A map $f$ in $\mathcal{A}(\dscr)$ is called a weak equivalence or fibration
if $f_\infty$ and each $f_k$ is 
(in $dg \qq \leftmod$ and $dg \qq W \leftmod$
respectively). This defines a monoidal
model structure on the category $\mathcal{A} (\dscr)$.

\begin{theorem}
There is a zig-zag of monoidal Quillen equivalences
between the model category of dihedral spectra
$\dscr \mcal_\qq$ and the algebraic model for 
dihedral spectra
$\mathcal{A}(\dscr)$.
\end{theorem}
\begin{proof}
We first use Lemma \ref{lem:localtomodules} to replace
$\dscr \mcal_\qq$ by $S_\dscr \leftmod$, 
a category with every object fibrant.
Theorem \ref{thm:topmorita} moves us to a 
category of right modules over a ringoid $\ecal_{top}$ (enriched over
$Sp^\Sigma_+$).
We then alter this ringoid to obtain $\ecal_t$,
a ringoid over $dg \qq \leftmod$, see
Theorem \ref{thm:EtopisEt}. 
On the algebraic side we have $\mathcal{A}(\dscr)$
which is equivalent to $\rightmod \ecal_a$
by Theorem \ref{thm:algmorita}.
Theorem \ref{thm:ETtoEA} proves that the $dg \qq$-ringoids
$\ecal_t$ and $\ecal_a$
are quasi-ismorphic, completing the argument.
\end{proof}

This paper exists to take the homotopy level classification
of dihedral spectra in \cite{gre98a} and lift it to
the level of model categories. This strengthens the result
and allows detailed consideration of monoidal structures. 
The proof of the theorem above is an adaptation of the results of
\cite[Chapter 4]{barnes}, which itself follows the 
basic plan of \cite{greshi}. 
The new points are the comparisons of the ringoids
and the construction of a model structure on 
$\mathcal{A}(\dscr)$. This classification
is possible since the homotopy of the
endomorphism ringoid $\ecal_{top}$ is concentrated in degree zero.
Consequently, one could have followed the method of
\cite[Example 5.1.2]{ss03stabmodcat}
and replaced $\ecal_{top}$ by an Eilenberg-MacLane
ringoid $\underline{\h} \ecal_a$ and then the results of 
that paper would prove that 
$\rightmod \ecal_{top}$, $\rightmod \underline{\h} \ecal_a$
and $\rightmod \ecal_a$ are Quillen equivalent. 
This alternative method would not be compatible with the monoidal structures. 

\begin{corollary}
The above zig-zag induces a
zig-zag of Quillen equivalences between the category
of algebras in $S_\dscr \leftmod$ and the category of
algebras in $\mathcal{A}(\dscr)$.
\end{corollary}

Let $i \co C_0 \ecal_t \to \ecal_t $ and 
$p \co C_0 \ecal_t \to \h_* \ecal_t$ be the maps constructed in 
Theorem \ref{thm:ETtoEA}. 
Recall from \cite[Theorem 4.1]{greshi} and \cite[Corollary 2.16]{shiHZ}
the composites  $\hom(\gcal_t, -) \circ D \circ \phi^*N \circ \widetilde{\qq}$
and $U \circ L'\cofrep \circ R \circ (\cofrep-) \otimes_{\ecal_t} \gcal_t$.
These are the derived composite functors of 
Theorem \ref{thm:EtopisEt}.
Cofibrant replacements ($\cofrep$) are not needed in the first of these 
composites as we are working rationally. The functor $L'$ is the 
alteration of $L$ from \cite{shiHZ} to modules over a ringoid
(\cite[Theorem 6.5]{ss03stabmodcat}).
We will then $L'$ so that it acts on categories
of algebras, we call this $L''$.

\begin{corollary}
Let $\Theta$ be the derived functor
$$(-) \otimes_{\ecal_a} \gcal_a \circ (\cofrep -) \square_{C_0 \ecal_t} \ecal_a
\circ i^* \circ \hom(\gcal_t, -) \circ
D \circ \phi^*N \circ \widetilde{\qq} \circ \hom(\gcal_{top}, -) $$ 
from $\dscr \mcal_\qq$ to $\mathcal{A} (\dscr)$.
Then for each $S_\dscr$-algebra $A$ there is a zig-zag of Quillen equivalences
between $A \leftmod$ and $\Theta A \leftmod$.
Let $\hh$ be the derived functor
$$
(\cofrep -) \smashprod_{\ecal_{top}} \gcal_{top} \circ 
U \circ L''\cofrep \circ R \circ (-) \otimes_{\ecal_t} \gcal_t \circ
(\cofrep -) \square_{C_0 \ecal_t} \ecal_t \circ p^* \circ \hom(\gcal_{a}, -) $$
from $\mathcal{A} (\dscr)$ to $\dscr \mcal_\qq$.
Then for each algebra object $B$ in $\mathcal{A}(\dscr)$ 
there is a zig-zag of Quillen equivalences
between $B \leftmod$ and $\hh B \leftmod$.
\end{corollary}

\section{The group $O(2)$}
We will use the notation $D^h_{2n}$ to represent the dihedral subgroup of order
$2n$ containing $h$, an element of $O(2) \setminus SO(2)$.
The closed subgroups of $O(2)$ are $O(2)$, $SO(2)$, the finite dihedral
groups $D^h_{2n}$ for each $h$ and the cyclic groups $C_n$ ($n \geqslant 1$).
We will always have $W=O(2)/SO(2)$.
Let $H \leqslant O(2)$, $N_{O(2)}(H)$ is the normaliser in $O(2)$ of $H$, 
it is the largest subgroup of $O(2)$ in which $H$ is normal.
The Weyl-group of $H$ in $O(2)$ is $W_{O(2)}(H):=N_{O(2)}(H)/H$.
The normaliser of $D^h_{2n}$ in $O(2)$ is $D^h_{4n}$, 
thus the Weyl group of $D_{2n}^h$ is $W$.
The cyclic groups are normal, hence the 
Weyl group of $C_n$ is $O(2)/C_n \cong O(2)$.
The Weyl group of $SO(2)$ is again $W$. 

Recall the following material from 
\cite[Chapter V, Section 2]{lms86}. 
Define $\mathcal{F} O(2)$ to be the set 
of those subgroups of $O(2)$
with finite index in their normaliser
(or equally, with finite Weyl-group),
equipped with the Hausdorff topology.
This is an $O(2)$-space via the conjugation 
action of $O(2)$ on its subgroups. 
This space is of interest due to
tom Dieck's ring isomorphism:
$$A(O(2)) \otimes \qq :=[S^0,S^0]^{O(2)} \otimes \qq \overset{\cong}{\longrightarrow} 
C(\mathcal{F} O(2)/O(2), \qq)$$
where $C(\mathcal{F} O(2)/O(2), \qq)$ is the ring of continuous maps
from $\mathcal{F} O(2)$ to $\qq$, considered as 
a discrete space. We draw $\mathcal{F} O(2)/O(2)$
below as Figure \ref{phig}.
We will sometimes write $D_{2n}$ for $(D^h_{2n})$, the conjugacy class of
$D^h_{2n}$. The point $O(2)$ is a limit point
of this space.

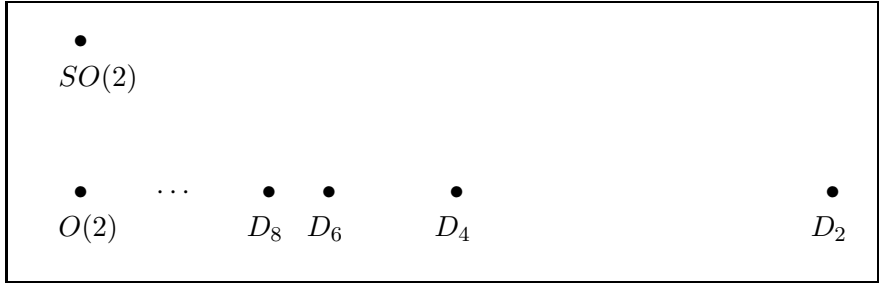
\begin{figure}[!hbt]
\begin{center}
\setlength{\unitlength}{1cm}
\framebox[0.8\textwidth]{
\begin{picture}(6,3.5)(2,1)

\put(0,2){$\bullet$}

\put(0,4){$\bullet$}

\put(10,2){$\bullet$}
\put(5,2){$\bullet$}
\put(2.5,2){$\bullet$}
\put(3.3,2){$\bullet$}
\put(1.1,2.0){$\cdots$}

\put(-0.2,1.5){$O(2)$}
\put(-0.2,3.5){$SO(2)$}

\put(9.8,1.5){$D_{2}$}
\put(4.8,1.5){$D_{4}$}
\put(3.1,1.5){$D_{6}$}
\put(2.3,1.5){$D_{8}$}

\end{picture}}
\end{center}
\caption{\label{phig} $\mathcal{F} O(2)/O(2)$.}
\end{figure}

\begin{definition}
Define $\cscr$\index{C@$\cscr$} to be the set consisting of the cyclic groups and $SO(2)$, 
this is a family of subgroups (that is, $\cscr$ is closed under conjugation and subgroups).
Let $\dscr$ be the complement of $\cscr$ in the set of all (closed) subgroups
of $O(2)$.
\end{definition}

We define idempotents of $C(\mathcal{F} O(2)/O(2), \qq)$ as follows:
$e_\cscr$ is the characteristic function of $SO(2)$, $e_\dscr =e_\cscr-1$
and $e_n$ is the characteristic function of $D_n$ for each $n \geqslant 1$,
note that $e_\dscr * e_n =e_n$.

\begin{lemma}
The rational Burnside ring of $O(2)$ is 
$\qq {e_\cscr} \oplus \qq {e_\dscr} \oplus \bigoplus_{n \geqslant 1} \qq {e_n}$, 
with multiplication defined by the multiplication of the idempotents. 
\end{lemma}

\section{Dihedral Spectra}
Following the work of \cite[Chapter 3]{barnes} we have 
the following constructions.
The category $O(2) \mcal_\qq$ is the category of 
$O(2)$-equivariant $S$-modules (\cite{mm02})
localised at the spectrum $S^0_\mcal \qq$. 
This spectrum is the cofibre of a map 
$g \co \bigvee_R \cofrep S \to \bigvee_F \cofrep S$
where $0 \to R \overset{f}{\to} F \to \qq$
is a free resolution of $\qq$ over $\zz$ 
and $F = \oplus_{q \in \qq} \zz$.
The map $g$ is a representative for the homotopy class
$f \otimes \id \co R \otimes \pi_*^{O(2)}(S) \to F \otimes \pi_*^{O(2)}(S)$. 
We have a map $\cofrep S \to S^0_\mcal \qq$ corresponding to
the inclusion of $\zz$ into the 1-factor of $F$. 
This map induces an isomorphism 
$\pi_*^H(\cofrep S \smashprod X) \otimes \qq \to 
\pi_*^H(S^0_\mcal \qq \smashprod X)$ for any spectrum $X$. 
The cofibrations of $O(2) \mcal_\qq$ are the same as for 
$O(2) \mcal$ and the weak equivalences $O(2) \mcal_\qq$ are
those maps $f$ such that $\pi_*^H(f) \otimes \qq$ is an isomorphism
of graded groups for all $H$.
There is a strong
symmetric monoidal Quillen equivalence
$$
\Delta : O(2) \mcal_\qq 
\overrightarrow{\longleftarrow} 
L_{EW_+} O(2) \mcal_\qq \times L_{S^{\infty \delta}} O(2) \mcal_\qq
: \prod .
$$
This result was proven at 
the homotopy level in \cite{gre98a}.
Note that $EW_+$ is rationally equivalent to 
$e_\cscr S$ and $S^{\infty \delta}$ is rationally
equivalent to $e_\dscr S$. 
We write $\dscr \mcal_\qq$ for $L_{S^{\infty \delta}} O(2) \mcal_\qq$.

We now examine the category $\dscr \mcal_\qq$ in more detail.
The weak equivalences are those maps $f$ such that 
$f \smashprod \id_{S^{\infty \delta}}$ induces an isomorphism
of rational homotopy groups, or equally, such that
$e_\dscr \pi_*^H(f) \otimes \qq$ is an isomorphism for all $H$.
The cofibrations of this model
category are the same as for $O(2) \mcal$. 
If $X$ is fibrant in this model structure, then $X$ is $S^{\infty \delta}$-local
and hence $\pi_*^H(X)=0$ for all $H \in \cscr$. Furthermore
all the homotopy groups of such an $X$ are rational vector spaces. 

\begin{lemma}\label{lem:localtomodules}
There is a commutative $S$-algebra, $S_\dscr$, such that 
$S \to S_\dscr$ is a weak equivalence in $\dscr \mcal_\qq$
and the adjunction below is a weak equivalence.
$$- \smashprod S_\dscr : \dscr \mcal_\qq \overrightarrow{\longleftarrow} 
S_\dscr \leftmod : U $$
\end{lemma}
\begin{proof}
The existence of $S_\dscr$ comes from the fact that one can localise
$S$ at a cell-$S$-module (in this case $S^0_\mcal \qq \smashprod S^{\infty \delta}$)
and obtain a commutative $S$-algebra, $S_\dscr$, such that the unit map is a 
$S^0_\mcal \qq \smashprod {S^{\infty \delta}}$-equivalence
(see \cite[Chapter VIII, Theorem 2.2]{ekmm97}). 

We claim that $S^0_\mcal \qq \smashprod {S^{\infty \delta}}$
and $S_\dscr$ are $\pi_*$-isomorphic. 
The construction of $S_\dscr$ comes with a 
$\pi_*$-isomorphism 
$S^0_\mcal \qq \smashprod {S^{\infty \delta}} \to 
S^0_\mcal \qq \smashprod {S^{\infty \delta}} \smashprod S_\dscr$.
This last term is $\pi_*$-isomorphic to 
$S^0_\mcal \qq \smashprod e_\dscr S_\dscr$.
The spectra
$S^0_\mcal \qq \smashprod S_\dscr$
and 
$(S^0_\mcal \qq \smashprod e_\cscr S_\dscr) \vee
(S^0_\mcal \qq \smashprod e_\dscr S_\dscr)$
are $\pi_*$-isomorphic.
Since $S_\dscr$ is $e_\dscr S$-local, 
it follows that $S^0_\mcal \qq \smashprod e_\cscr S_\dscr$ 
is $\pi_*$-isomorphic to a point.
Since $S_\dscr$ is $S^0_\mcal \qq$-local,
$S_\dscr$ and $S^0_\mcal \qq \smashprod S_\dscr$
are $\pi_*$-isomorphic and we have proven 
our claim.
It then follows by a standard argument that 
all $S_\dscr$-modules are 
$S^0_\mcal \qq \smashprod {S^{\infty \delta}}$-local. 

Now we can prove that the above adjunction is a 
strong symmetric monoidal Quillen
equivalence. Note that 
the weak equivalences and fibrations of $S_\dscr \leftmod $
are defined in terms of their underlying maps in $O(2) \mcal$. 
The left adjoint preserves cofibrations and it takes
acyclic cofibrations to 
$S^0_\mcal \qq \smashprod {S^{\infty \delta}}$-equivalences
between $S_\dscr$-modules. Since such a module is 
$S^0_\mcal \qq \smashprod {S^{\infty \delta}}$-local, 
it follows that the left adjoint takes 
acyclic cofibrations to $\pi_*$-isomorphisms. 

The right adjoint preserves and detects all weak-equivalences, 
so we must prove that for $X$, cofibrant in $\dscr \mcal_\qq$,
the map $S \to X \smashprod S_\dscr$ is a weak-equivalence in 
$\dscr \mcal_\qq$. This is immediate, since smashing with a 
cofibrant object will preserve the 
$S^0_\mcal \qq \smashprod {S^{\infty \delta}}$-equivalence 
$S \to S_\dscr$. 
\end{proof}

Note that an $S_\dscr$-module $X$ has rational homotopy groups
and $\pi_*^{H}(X)=0$ for any $H \in \cscr$. 

\begin{lemma}
The model category $S_\dscr \leftmod $
is generated by   
$S_\dscr$ and the countably infinite  collection 
$\{ S_\dscr \smashprod \cofrep e_H O(2)/H_+  \ | \ H \in \dscr \setminus \{O(2) \} \}$.
Furthermore, these objects are compact. 
\end{lemma}
\begin{proof}
We must prove that if $X$ is an object of 
$S_\dscr \leftmod $ such that 
$[\sigma ,X]_*^{S_\dscr}=0$ (maps in the homotopy 
category of $S_\dscr \leftmod $)
for all $\sigma$ as above, 
then $X \to *$ is a $\pi_*$-isomorphism. 
As mentioned above $\pi_*^{H}(X)=0$ for any $H \in \cscr$ and 
since $[S_\dscr ,X]_*^{S_\dscr}=0$ we see that 
$\pi_*^{O(2)}(X)=0$.
So now we must consider a finite dihedral group $H$:
by \cite[Example C(i)]{gre98a}, $\pi_*^H(X)$ 
is given by 
$\oplus_{(K) \leqslant H} (e_K \pi_*^K(X))^{W_H K}$.
We have assumed that 
$[S_\dscr \smashprod \cofrep e_K O(2)/K_+ ,X]_*^{S_\dscr}=0$, 
for each finite dihedral $K$, 
but this is precisely the condition that 
$e_K \pi_*^K(X) = 0$ for each $K$. 
Thus $\pi_*^H(X)=0$ and our set generates the homotopy category. 
Compactness follows from the isomorphisms
$[\sigma_H, X]^{S_\dscr}_* \cong e_H \pi_*^H(X)$. 
\end{proof}

\begin{proposition}
The category of dihedral spectra is a spectral model category.
\end{proposition}
\begin{proof}
This takes a little work and we must use the 
category of ${O(2)}$-equivariant orthogonal spectra.
We begin with the the composite functor
$i_* \varepsilon^*_{O(2)} \co \iscr \sscr_+ \to {O(2)} \iscr \sscr_+$
as defined in \cite[Chapter V, Proposition 3.4]{mm02}, which
states that this functor is part of a Quillen pair $(i_* \varepsilon^*_{O(2)}, (i^* (-))^{O(2)})$.
This is a strong symmetric monoidal adjunction, as
noted on \cite[Chapter V, Page 80]{mm02}.
Now we use the following diagram:
$$ \xymatrix@C+0.6cm{
Sp^\Sigma_+
\ar@<+0.4ex>[r]^(0.6){\mathbb{P} \circ |-|} &
\iscr \sscr_+
\ar@<+0.4ex>[r]^(0.5){i_* \varepsilon^{*}_{O(2)}}
\ar@<+0.4ex>[l]^(0.4){\sing \circ \mathbb{U}} &
O(2) \iscr \sscr_+
\ar@<+0.4ex>[r]^(0.6){\nn }
\ar@<+0.4ex>[l]^(0.5){(i^* (-))^{O(2)} } &
O(2) \mcal
\ar@<+0.4ex>[r]^(0.5){S_\dscr \smashprod (-)}
\ar@<+0.4ex>[l]^(0.4){\nn^\# } &
{S_\dscr} \leftmod
\ar@<+0.4ex>[l]^(0.5){U}
 } $$
For a pair of $S_\dscr$-modules $X$ and $Y$, the
$Sp^\Sigma_+$-function object
$\underhom (X,Y) \in Sp^\Sigma_+$
is given by $\sing \mathbb{U} (i^* \nn^\# U F_{S_\dscr}(X,Y))^{O(2)}$.
\end{proof}

So now we have a spectral model category and a set of 
compact, fibrant generators such that every non-unit object is fibrant. 
Take the closure of this set under the smash product, call it 
$\gcal_{top}$. Take the subcategory of $S_\dscr \leftmod$
on object set $\gcal_{top}$ considered as a category enriched
over symmetric spectra, we denote this spectral category by $\ecal_{top}$.
We will sometimes call an enriched category a ringoid. 
A right module over $\ecal_{top}$
is a contravariant enriched functor $M \co \ecal_{top} \to Sp^\Sigma_+$,
the category of such functors is denoted $\rightmod \ecal_{top}$.
Later we will use other categories in place of 
$Sp^\Sigma_+$, such as $dg \qq$-modules.
The category of right modules over $\ecal_{top}$ has a model 
structure with weak equivalences and fibrations 
defined objectwise in $Sp^\Sigma_+$,
see \cite[Subsection 3.3]{ss03stabmodcat} for more details.


\begin{theorem}\label{thm:topmorita}
The adjunction below is a Quillen equivalence with
strong symmetric monoidal left adjoint and
lax symmetric monoidal right adjoint.
$$(-) \otimes_{\ecal_{top} } \gcal_{top} :
\rightmod \ecal_{top}
\overrightarrow{\longleftarrow}
S_\dscr \leftmod :
\underhom(\gcal_{top}, -)$$
\end{theorem}
\begin{proof}
This follows from \cite[Theorem 3.9.3]{ss03stabmodcat}
and \cite[Theorem 4.1]{greshi}, see also
\cite[Theorem 9.1.2]{barnes}.
\end{proof}

\begin{theorem}\label{thm:EtopisEt}
There is a zig-zag of monoidal Quillen equivalences between
$\rightmod \ecal_{top}$ (enriched over
$Sp^\Sigma_+$) and a category
$\rightmod \ecal_{t}$
(enriched over $dg \qq \leftmod$).
This zig-zag induces
an isomorphism of monoidal graded $\qq$-categories:
$\pi_*(\ecal_{top}) \cong \h_* \ecal_t$.
\end{theorem}
\begin{proof}
This is contained in the proof of \cite[Theorem 4.1]{greshi}
which is based on \cite[Corollary 2.16]{shiHZ}.
Further details can be found in 
\cite[Section 9.3]{barnes}, 
noting that $\ecal_t$ is, in fact, a monoidal $dg \qq$-category. 
We describe the adjoint pairs below.
\end{proof}

The zig-zag consists of the following functors 
applied to modules over ringoids.
The free simplicial $\qq$-module functor induces an 
adjunction:
$\widetilde{\qq} : Sp^\Sigma \overrightarrow{\longleftarrow}
Sp^\Sigma(\sqq \leftmod) : U$
Normalisation of simplicial $\qq$-modules
gives an adjunction:
$$\phi^* N : Sp^\Sigma(\sqq \leftmod) \overrightarrow{\longleftarrow}
Sp^\Sigma(dg \qq \leftmod_+) : L.$$
Finally there is the adjoint pair
$D : Sp^\Sigma(dg \qq \leftmod_+)
\overrightarrow{\longleftarrow}
dg \qq \leftmod : R$,
where the functor $R$
takes a chain complex $Y$
to the symmetric spectrum with
$RY_n = C_0 (Y \otimes \qq[m])$.

These adjoint pairs induce a zig-zag of Quillen equivalences
between $\rightmod \ecal_{top}$
and $\rightmod D \phi^* N \widetilde{\qq} \ecal_{top}$.
One also needs to apply an adjoint pair
similar to that of Theorem \ref{thm:topmorita}.
Consider the set 
$\{ S_\dscr \smashprod \cofrep e_H O(2)/H_+  \ | \ H \in \dscr \setminus \{O(2) \} \}$.
Apply $\hom( \gcal_{top},-)$ and the above functors
to this set to obtain a collection of objects in 
$\rightmod D \phi^* N \widetilde{\qq} \ecal_{top}^H$.
Take cofibrant replacements and consider all such 
products of these objects. This set is $\gcal_t$
and gives the $dg \qq$-ringoid $\ecal_t$.
We then have a Quillen equivalence between 
$\rightmod D \phi^* N \widetilde{\qq} \ecal_{top}^H$
and $\rightmod \ecal_t$.
Note that the functor $D$, from \cite{shiHZ}, is symmetric monoidal. 
There has been some confusion over this issue, which has now been resolved by 
a detailed note by Neil Strickland: \cite{Dsymmon}.
Thus we have an algebraic model for dihedral spectra,
albeit one which is not particularly explicit. 
We spend the rest of this paper constructing a better
algebraic model and proving that is is monoidally Quillen
equivalent to $\rightmod \ecal_t$.

\section{The Algebraic Model}
We take the work of \cite{gre98a} and consider the algebraic model for the
homotopy category of dihedral spectra. We give this category 
a monoidal model structure and then replace it 
by modules over a $dg \qq$-ringoid.

\begin{definition}
An object $V$ of the algebraic model $\mathcal{A}(\dscr)$
consists of a differential graded rational vector space $V_{\infty}$
and $dg \qq W$-modules $V_k$ for each $k \geqslant 1$,
with a map of $dg \qq W$-modules
$\sigma_V \co V_\infty \to \colim_n \prod_{k \geqslant n} V_k$.
We will always consider a $dg \qq$-module
as a $dg \qq W$-module with trivial $W$-action. 

A map $f \co V \to V'$ in this category consists of 
a map $f_\infty \co V_{\infty} \to V_{\infty}'$
in $dg \qq \leftmod$ and 
$dg \qq W$-module maps 
$f_k \co V_k \to V_k'$
making the square below commute. 
We often write $V(\text{end})$ for 
$\colim_n \prod_{k \geqslant n} V_k$.
$$\xymatrix@C+0.5cm{
V_{\infty} \ar[r]^(0.4){\sigma_V} \ar[d]_{f_\infty} &
\colim_n \prod_{k \geqslant n} V_k 
\ar[d]^{\colim_n \prod_{k \geqslant n} f_k} \\
V_{\infty}' \ar[r]^(0.4){\sigma_{V'}} &
\colim_n \prod_{k \geqslant n} V_k' 
}$$
\end{definition}

The following definition and theorem 
are taken from \cite{gre98a}.
Note that for any compact Lie group $G$ and closed subgroup $H$, 
the action of $N_G H/H$ on $G/H$
induces an action of $N_G H/H$ on 
$[G/H_+,X]^G_* \cong \pi_*^H(X)$. 
\begin{definition}\label{def:dihedralmackey}
For an $O(2)$-spectrum $X$ with rational homotopy groups, let 
$\underline{\pi}_*(X)$ denote the following object of 
$\mathcal{A}(\dscr)$ with trivial differential. 
Let $k \geqslant 1$ and take $H$ a dihedral group with 
$|H| =2k$, then
$\underline{\pi}_*(X)_k = e_H \pi_*^H(X) \otimes \qq$. 
Let $f_n = e_\dscr - \Sigma_{i=1}^n e_i$ and set
$\underline{\pi}_*(X)_\infty = \colim_n ( f_n \pi_*^{O(2)}(X) \otimes \qq)$. 
The structure map is induced by 
the collection of maps
$$f_n \pi_*^{O(2)}(X) \otimes \qq \longrightarrow (\iota_H)_*(f_n) \pi_*^{H}(X) \otimes \qq
\to e_H (\iota_H)_* (f_n) \pi_*^{H}(X) \otimes \qq$$
where the first arrow arises from the forgetful functor $\iota^*_H$
($\iota_H$ is the inclusion $H \to O(2)$) and the second is applying $e_H$.
\end{definition}

This construction defines a functor 
$\underline{\pi}_* \co \ho S_\dscr \leftmod \to \mathcal{A}(\dscr)$,
using the forgetful functor 
$\ho S_\dscr \leftmod \to \ho O(2) \mcal$.
Thus one has a map of $\qq$-modules
$e_\dscr[X,Y]^{O(2)} \to 
\hom_{\mathcal{A}(\dscr)} ( \underline{\pi}_*( X),
\underline{\pi}_*(Y) )$. 

\begin{theorem}\label{thm:sesmackey}
For $X$ and $Y$, $O(2)$ spectra with rational homotopy groups, 
there is a short exact sequence as below.
$$
0 \to 
\ext_{\mathcal{A}(\dscr)}( \underline{\pi}_*( \Sigma X),
\underline{\pi}_*(Y) ) \to 
e_\dscr[X,Y]^{O(2)} \to 
\hom_{\mathcal{A}(\dscr)} ( \underline{\pi}_*( X),
\underline{\pi}_*(Y) ) \to 0
$$
\end{theorem}

The homotopy category of $\mathcal{A}(\dscr)$ agrees
with the algebraic model for dihedral spectra 
in \cite{gre98a}.
We now introduce a construction that we will make much use of, 
we will soon see that this is an explicit description
of the `global sections' of an object of $\mathcal{A}(\dscr)$.

\begin{definition}
Let $n \geqslant 1$ and take $V \in \mathcal{A} (\dscr)$. 
Then $\bignplus_N V$ is defined as the following pullback
in the category of $dg \qq W$-modules. 
$$\xymatrix{
\bignplus_N V
\ar[r]
\ar[d] &
\prod_{k \geqslant N} V_k
\ar[d] \\
V_\infty 
\ar[r] & \colim_n \prod_{k \geqslant n} V_k
}$$
We also have $\bignplus_N^W V: =(\bignplus_N V)^W$ the $W$-fixed points of 
$\bignplus_N V$.
\end{definition}

Note that $V(\text{end})^W = \colim_n \prod_{k \geqslant n} V_k^W$
and that the structure map of $V$ induces a map 
$V_\infty \to V(\text{end})^W$. So we can construct 
$\bignplus_N^W V$ in terms of a pullback of $dg \qq$-modules. 
The notation $\bignplus_N$ is to make the reader think 
of some combination of a direct product and a direct sum.
Indeed if $V_\infty=0$, then $\bignplus_N V = \bigoplus_{k \geqslant n} V_k$. 

Let $\pscr$ be the space $\fcal O(2)/O(2) \setminus \{ SO(2) \}$
($\pscr$ for points) and let $\ocal$ be the constant
sheaf of $\qq$, this is a sheaf of rings.
To specify an $\ocal$-module $M$ one only needs to give
the stalks at the points $k$ and $\infty$ and a 
$\qq$ map $M_\infty \to M(\text{end})$. 
The global sections of $M$ are then given by 
$\bignplus_1 M$. 
One can then consider $W$-equivariant objects in 
$\ocal \leftmod$, we denote this category by $W \ocal \leftmod$. 
The category $\mathcal{A}(\dscr)$ is a full subcategory of 
$W \ocal \leftmod$; we have an inclusion
functor $\text{inc} \co \mathcal{A}(\dscr) \to W \ocal \leftmod$
and this has a right adjoint: $\text{fix}$.
On an $W$-equivariant $\ocal$-module $V$,  $\text{fix}(V)_k = V_k$,
$\text{fix}(V)_\infty = V_\infty^W$ and the structure map
is $V_\infty^W \to V_\infty \to V(\text{end})$. 
Our definitions and constructions are, therefore, slight adjustments 
to the usual definitions of modules over a sheaf of rings.
 
\begin{lemma}
The category $\mathcal{A} (\dscr)$ contains all small limits and colimits.
\end{lemma}
\begin{proof}
Take some small diagram $V^i$ of objects of $\mathcal{A} (\dscr)$. 
Define $(\colim_i V^i)_\infty = \colim_i (V^i_\infty)$
and $(\colim_i V^i)_k = \colim_i (V^i_k)$. 
The map below induces (via the universal properties of colimits)
a structure map for $\colim_i V^i$.
$$
V^i_\infty \longrightarrow 
\colim_n \prod_{k \geqslant n} V_k^i \longrightarrow
\colim_n \prod_{k \geqslant n} \colim_i V_k^i 
$$
Limits are harder to define because we are working with a 
stalk-based description of a `sheaf'. 
Let $(\lim_i V^i)_k = \lim_i (V^i_k)$ and 
$(\lim_i V^i)_\infty = \colim_N \lim_i \bignplus_N^W V^i$.
The structure map is then as below. 
$$(\lim_i V^i)_\infty = 
\colim_N \lim_i {\bignplus}_N^W V^i \longrightarrow
\colim_N \lim_i \prod_{k \geqslant N} V^i_k = 
(\lim_i V^i)(\text{end})$$
\end{proof}

One could also describe the limit of some diagram $V^i$ 
in $\mathcal{A} (\dscr)$ as $\text{fix} \lim_i \text{inc } V^i$, 
where the limit on the right is taken in the category of $W \ocal$-modules. 
Let $M$ be a $dg \qq$ module, $R$ a $dg \qq W$-module
and $V \in \mathcal{A}(\dscr)$. Then we define
$i_k R$ to be that object of  $\mathcal{A}(\dscr)$
with $(i_k R)_\infty =0$, $(i_k R)_n = 0$ for $n \neq k$ and
$(i_k R)_k = R$. Let $p_k V = V_k$, an object of $dg \qq W \leftmod$
and $p_\infty V = V_\infty$, a 
$dg \qq$-module. The functor $i_\infty$ takes $M$
to that object of  $\mathcal{A}(\dscr)$ with 
$(i_\infty M)_\infty = M$ and $(i_\infty M)_k = 0$.
We set $cM$ to be the object of $\mathcal{A}(\dscr)$
with $cM_k = M = cM_\infty$ and structure map induced by the 
diagonal map $M \to \prod_{k \geqslant 1} M$. 

\begin{lemma}\label{lem:manyadjoints}
We have the following series of adjoint pairs.
$$
\begin{array}{rcl}
i_k : dg \qq W \leftmod & \overrightarrow{\longleftarrow} &
\mathcal{A}(\dscr) : p_k \\
p_k : \mathcal{A}(\dscr) & \overrightarrow{\longleftarrow} &
dg \qq W \leftmod  : i_k \\
p_\infty : \mathcal{A}(\dscr) & \overrightarrow{\longleftarrow} &
dg \qq \leftmod : i_\infty\\
c : dg \qq \leftmod & \overrightarrow{\longleftarrow} &
\mathcal{A}(\dscr) : \bignplus_1^W \\
\end{array}
$$
\end{lemma}
\begin{proof}
This is a routine check, we just wish to comment
that $(c, \bignplus_1^W)$ is the constant sheaf and global sections
adjunction. 
\end{proof}

\begin{lemma}
The category $\mathcal{A}(\dscr)$ is a closed symmetric monoidal category. 
Furthermore there is  
a symmetric monoidal adjunction as below.
$$c : dg \qq \leftmod \overrightarrow{\longleftarrow} 
\mathcal{A}(\dscr) : {\bignplus}^W_1$$
\end{lemma}
\begin{proof}
The category of $\ocal$-modules is closed symmetric monoidal, 
hence so is the category $ W \ocal \leftmod$.
That is, for $M$ and $N$ in $ W \ocal \leftmod$, 
$W$ acts diagonally on $(M \otimes_\ocal N)(U)$
and by conjugation on $\hom_\ocal(M,N)(U)$
(for $U$ an open subset of $\pscr$).
We `restrict' this structure to $\mathcal{A}(\dscr)$.
Take $A$, $B$ and $C$ in $\mathcal{A}(\dscr)$, 
the tensor product $A \otimes_\ocal B$ is in 
$\mathcal{A}(\dscr)$ and is given by:
$(A \otimes_\ocal B)_k = A_k \otimes_\qq B_k$, 
$(A \otimes_\ocal B)_\infty = A_\infty \otimes_\qq B_\infty $
and the structure map is given by the composite of the three maps below.
$$\begin{array}{rcl}
A_\infty \otimes B_\infty & \longrightarrow & 
\colim_n (\prod_{k \geqslant n} A_k) \otimes
\colim_n (\prod_{k \geqslant n} B_k) \\
\colim_n (\prod_{k \geqslant n} A_k \otimes
\prod_{k \geqslant n} B_k)
& \overset{\cong}{\longrightarrow} & 
\colim_n (\prod_{k \geqslant n} A_k) \otimes
\colim_n (\prod_{k \geqslant n} B_k) \\
\colim_n (\prod_{k \geqslant n} A_k \otimes
\prod_{k \geqslant n} B_k)
& \longrightarrow & 
\colim_n \prod_{k \geqslant n} (A_k \otimes B_k)
\end{array}$$
We now have a series of natural isomorphisms, 
where we suppress notation for $\text{inc}$.
$$\begin{array}{rcl}
\mathcal{A}(\dscr) (A \otimes_\ocal B, C) & = &
W \ocal \leftmod (A \otimes_\ocal B, C) \\
& \cong &
W \ocal \leftmod (A , \hom_\ocal (B, C)) \\
& \cong &
\mathcal{A}(\dscr) (A, \text{fix} \hom_\ocal (B, C)) 
\end{array}$$
Thus we take the internal homomorphism object for 
$\mathcal{A}(\dscr)$ to be $\text{fix} \hom_\ocal (B, C)$.
It is routine to prove that $(c, \bignplus_1^W)$ is a strong 
symmetric monoidal adjunction.
\end{proof}

Our model structure is an alteration of the
flat model structure from \cite{hov01sheaves},
noting that every sheaf on $\pscr$ is automatically flasque
so the fibrations are precisely the stalk-wise surjections.
Let $I_\qq$ and $J_\qq$ denote the sets of generating 
cofibrations and 
acyclic cofibrations for the projective model structure on 
$dg \qq \leftmod$, see \cite[Section 2.3]{hov99}. 
Similarly we have $I_{\qq W}$ and $J_{\qq W}$ for 
$dg \qq W \leftmod$.

\begin{theorem}
Define a map $f$ in $\mathcal{A}(\dscr)$ to be 
a weak equivalence or fibration
if $f_\infty$ and each $f_k$ is 
(in $dg \qq \leftmod$ and $dg \qq W \leftmod$
respectively). This defines a monoidal
model structure on the category $\mathcal{A} (\dscr)$.
Furthermore this model structure is cofibrantly generated, 
proper and satisfies the monoid
axiom. The generating cofibrations, $I$, are the collections
$cI_\qq$ and $i_k I_{\qq W}$, for $k \geqslant 1$. The generating
acyclic cofibrations, $J$, are 
$cJ_\qq$ and $i_k J_{\qq W}$, for $k \geqslant 1$. 
\end{theorem}
\begin{proof}
We prove that the above sets of maps define a model structure
using \cite[Theorem 2.1.19]{hov99}. 
We prove in Lemma \ref{lem:globseccolim} below that ${\bignplus}^W_1$
preserves filtered colimits. From this it follows
that the domains of the generating cofibrations and
acyclic cofibrations are small. 

We identify the maps with the right lifting property with 
respect to $I$. Let $f \co A \to B$ be such a map, 
using the adjunctions of 
Lemma \ref{lem:manyadjoints} it follows that each $f_k \co A_k \to B_k$
must be a surjection and a homology isomorphism,
as must $\bignplus_1^W f \co \bignplus_1^W A \to \bignplus_1^W B$. 
In turn, each $f_k^W$ is a homology isomorphism 
and a surjection so 
$\bignplus_n^W f$ is a surjection and a homology isomorphism
for each $n \geqslant 1$.. 
Taking colimits over $n$ we see that $f_\infty$ is a 
surjection and homology isomorphism. 

Now take a map $f \co A \to B$ that is stalk-wise 
a surjection and homology
isomorphism. Since $\bignplus_1^W A$ and $\bignplus_1^W B$ 
are homotopy limits in $dg \qq \leftmod$ 
we see that $\bignplus_1^W f$ is a homology isomorphism. 
It is also a surjection by similar arguments to that of
Lemma \ref{lem:globseccolim}. Hence $f$ has the right
lifting property with respect to $I$.
Similarly the maps
with the right lifting property with 
respect to $J$ are precisely the stalk-wise surjections. 

To complete the proof that we have a cofibrantly generated model category we 
must prove that transfinite compositions of pushouts of elements
of $J$ are weak equivalences that are $I$-cofibrations. 
Showing that such maps are $I$-cofibrations is routine. 
The maps in $J$ are stalk-wise injective homology isomorphisms,
such maps are preserved by pushouts and transfinite compositions
and so the result follows. 

Checking the pushout product axiom for the generators is routine.
The monoid axiom holds because the functors $p_k$ and 
$p_\infty$ are strong monoidal left adjoints such that if 
each $p_k f$ and $p_\infty f$ are weak equivalences 
then $f$ is a weak equivalence in $\mathcal{A}(D)$. 

Left properness is immediate because colimits are defined stalk-wise. 
Right properness only requires work over $\infty$, let $P$
be the pull back of $X \to Z \leftarrow Y$, with 
$X \to Z$ a weak equivalence and $Y \to Z$ a fibration. 
Applying the right Quillen functor $\bignplus_n^W$ 
gives a pullback diagram in $dg \qq \leftmod$, which is right proper.
Hence $\bignplus_n^W P \to \bignplus_n^W X$ is a weak equivalence
for each $n \geqslant 1$. Taking colimits over $n$
shows that $P_\infty \to X_\infty$ is a homology isomorphism.  
\end{proof}

\begin{lemma}\label{lem:globseccolim}
The functors
${\bignplus}_n$
preserves filtered colimits
for all $n \geqslant 1$.
\end{lemma}
\begin{proof}
One has a canonical map
$\colim_i \bignplus_1 V^i \to \bignplus_1 \colim_i V^i $.
For $X \in \mathcal{A}(\dscr)$, we write elements of 
$\bignplus_1 X$ as $(x_\infty, x_1, x_2, \dots)$, 
for $x_\infty \in X_\infty$ and $x_i \in X_i$, such that 
$x_\infty$ and $(x_1, x_2, \dots)$ agree in $X(\text{end})$. 
For each $n >1$ there is an isomorphism of $dg\qq W$-modules
$\bignplus_1 X \cong \bignplus_{n} X \oplus \bigoplus_{1 \leqslant k < n} X_k$,
which sends 
$(x_\infty, x_1, x_2, \dots)$ to the term
$((x_\infty, x_n, x_{n+1}, \dots),x_1,x_2, \dots, x_{n-1})$. 
Hence, for each $n > 1$ we have specified an isomorphism between 
$\colim_i \bignplus_1 V^i$ and 
$\colim_i \bignplus_{n} V^i \oplus \bigoplus_{1 \leqslant k < n} \colim_i V^i_k$.

It suffices to prove the result for $n=1$, 
let $\{ V^i \}_{i \in I}$ be a filtered diagram 
in $\mathcal{A}(\dscr)$.
Take some $(v_\infty, v_1, v_2, \dots)$ in 
$V^j$ which maps to zero in $\bignplus_1 \colim_i V^i$.  
Thus there is some $m \in I$, with a map $j \to m$ in $I$
such that $v_\infty$ is sent to zero in $V^m_\infty$.
Consider $(0, w_1, w_2, \dots) \in \bignplus_1 V^m$, the image of 
$(v_\infty, v_1, v_2, \dots)$. Since the term at infinity is zero, 
only finitely many $w_k$ are non-zero, let $N$ be such that
for $k \geqslant N$, $w_k=0$. 
Note that each $w_k$ is zero in $\colim_i V^i_k$.
Write $\bignplus_1 V^m$ as 
$\bignplus_{N} V^m \oplus \bigoplus_{1 \leqslant k < N} V^m_k$.
Thus $(0, w_1, w_2, \dots)$ is given by 
$((0,0, \dots), w_1, w_2, \dots, w_{N-1})$ in
$\bignplus_{N} V^m \oplus \bigoplus_{1 \leqslant k < N} V^m_k$.
The element $((0,0, \dots), w_1, w_2, \dots, w_{N-1})$ 
is clearly zero in the colimit, hence so are 
$(0, w_1, w_2, \dots)$ and $(v_\infty, v_1, v_2, \dots)$. 


Now we consider surjectivity, take 
$([z_\infty], [z_1], [z_2], ) \in \bignplus_1 \colim_i V^i $.
So there is some $j \in I$ with
$z_\infty \in V^j_\infty$. Pick a representative
$(y_1, y_2, \dots) \in V^j(\text{end})$
for the image of $z_\infty$.
Thus we have $(z_\infty, y_1, y_2, \dots) \in \bignplus_1 V^j$.
The difference between the image of 
$(z_\infty, y_1, y_2, \dots)$ in $\bignplus_1 \colim V^i$ and $([z_\infty], [z_1], [z_2], )$
has only finitely many non-zero terms.
We now proceed as for injectivity.
\end{proof}

\begin{corollary}
The adjunctions of Lemma \ref{lem:manyadjoints}
are strong symmetric monoidal Quillen pairs.
\end{corollary}

\begin{lemma}
The collection $i_k \qq W$ for $k \geqslant 1$
and $c \qq$ are a set of compact, 
cofibrant and fibrant generators for this category.
\end{lemma}
\begin{proof}
That these are cofibrant follows immediately from
the fact that $\qq$ and $\qq W$ are cofibrant in 
$dg \qq \leftmod $ and $dg \qq W \leftmod $ respectively. 

Take $V$ in $\mathcal{A}(\dscr)$, we must show that
$M \to 0$ is a weak equivalence if and only if
$[G,M]_*^{\mathcal{A}(\dscr)}=0$ for each object in our set.
First we know (since every object is fibrant) that
$[i_k \qq W, V]_*^{\mathcal{A}(\dscr)} \cong 
[\qq W, V_k]_*^{\qq W}$, where the right hand side is 
maps in $\ho dg \qq W \leftmod$.
Secondly $[c \qq, V]^{\mathcal{A}(\dscr)}_* \cong [\qq, \bignplus_1^W V ]^\qq_*$,
with the right hand side is maps in the 
homotopy category of $dg \qq \leftmod$. 
Since $\qq$ generates $dg \qq \leftmod$
and $\qq W$ generates $dg \qq W \leftmod$, 
we know that $V_k$ is acyclic for all $k$ and 
that $\bignplus_1^W V$ is acyclic. 
The following isomorphism then tells us that 
$\bignplus_{N}^W V $ is acyclic for each $N$.
$${\bignplus}_1^W V \cong {\bignplus}_{N}^W V  
\oplus \bigoplus_{1 \leqslant k < N} V_k^W
$$
So now we have the following series of isomorphisms. 
$H_* V_\infty \cong  H_* (\colim_N {\bignplus}_{N}^W V )
\cong \colim_N H_* ({\bignplus}_{N}^W V ) =0$.
\end{proof}

Let $\ecal_a$ denote the full subcategory of $\mathcal{A}(\dscr)$
on tensor products of the generators $\gcal_a$, considered as a category
enriched over $dg \qq \leftmod$. 
This result below follows from \cite[Proposition 3.6]{greshi}.
\begin{theorem}\label{thm:algmorita}
The adjunction below is a Quillen equivalence with
strong symmetric monoidal left adjoint and
lax symmetric monoidal right adjoint.
$$(-) \otimes_{\ecal_a} \gcal_a :
\rightmod \ecal_a
\overrightarrow{\longleftarrow}
\mathcal{A}(\dscr) :
\underhom(\gcal_a, -)$$
\end{theorem}

\section{Comparison of Ringoids}

\begin{proposition}\label{prop:homobjcalc}
For $H$ a dihedral subgroup of $O(2)$ of order $2k$, 
and $i \geqslant 1$, 
let $\sigma_H^i= (S_\dscr \smashprod \cofrep e_H O(2)/H_+ )^i$, 
with the smash product taken in the category of $S_\dscr$-modules. 
Then
$\underline{\pi}_*(S_\dscr)  =  c \qq$ and 
$\underline{\pi}_*(\sigma_H^i)  =  i_k \qq^{\otimes_i} W$.
\end{proposition}
\begin{proof}
We already know the homotopy groups of $S_\dscr$:
$\pi_*^K(S_\dscr) \cong \iota_K^*(e_\dscr) \pi_*^K(S) \otimes \qq$.
For $\pi_*^K(\sigma_H^i)$ we use the following facts:
$\Phi^H X \simeq (e_H X)^H$, $\Phi^H (X \smashprod Y) \simeq 
\Phi^H (X) \smashprod \Phi^H (Y)$, $\Phi^H \Sigma^\infty A \simeq \Sigma^\infty A^H$
and $(O(2)/H)^H = W$, where
$X$ and $Y$ are $O(2)$-spectra and $Y$ is an $O(2)$-space. 
Thus  
$\pi_*^K(\sigma_H^i) \cong \iota_*(e_H) \pi_*^K( O(2)/H_+^{\smashprod_i} ) \otimes \qq$.
\end{proof}

Recall that $\h_*\ecal_t \cong \pi_* \ecal_{top}$
as monoidal enriched categories, so we work with 
$\pi_* \ecal_{top}$. This is a category enriched 
over graded rational vector spaces and 
is isomorphic to the full $g \qq$-subcategory of $\ho S_\dscr \leftmod$ 
on object set $\gcal_{top}$. Furthermore this is an isomorphism 
of monoidal ringoids. Thus we calculate $[\sigma, \sigma']^{S_\dscr}_*$
for each pair of objects in $\gcal_{top}$ and we have an isomorphism from this 
graded $\qq$-vector space to $\h_*\ecal_t (\sigma, \sigma')$. 
The following proposition is immediate from the calculation above and 
Theorem \ref{thm:sesmackey}.

\begin{proposition}
Let $i,j,k,m \geqslant 1$ and 
let $H$ and $K$ be finite dihedral groups with $|H|=2k$ and $|K|=2m$.
Then the functor $\underline{\pi}_*$
from Definition \ref{def:dihedralmackey} gives isomorphisms as below.
$$
\begin{array}{rcl}
{[S_\dscr, S_\dscr]^{S_\dscr}_* }& \overset{\cong}{\longrightarrow} & 
{[c \qq , c \qq]^{\mathcal{A}(\dscr)}_* } \\
{[\sigma_H^i , S_\dscr]^{S_\dscr}_* }& \overset{\cong}{\longrightarrow} & 
{[i_k \qq W^{\otimes_i} , c \qq]^{\mathcal{A}(\dscr)}_* } \\
{[S_\dscr, \sigma_H^i ]^{S_\dscr}_* }& \overset{\cong}{\longrightarrow} & 
{[c \qq, i_k \qq W^{\otimes_i}  ]^{\mathcal{A}(\dscr)}_* } \\
{[\sigma_H^j , \sigma_H^i ]^{S_\dscr}_* }& \overset{\cong}{\longrightarrow}& 
{[i_k \qq W^{\otimes_j} , i_k \qq W^{\otimes_i}  ]^{\mathcal{A}(\dscr)}_* } \\
{[\sigma_K^j , \sigma_H^i ]^{S_\dscr}_* }& \overset{\cong}{\longrightarrow} & 
{[i_m \qq W^{\otimes_j} , i_k \qq W^{\otimes_i}  ]^{\mathcal{A}(\dscr)}_* } \\
\end{array}
$$ 
\end{proposition}

\begin{proposition}
Let $i,j,k \geqslant 1$, then 
$$
\begin{array}{rcl}
{[c \qq , c \qq]^{\mathcal{A}(\dscr)}_* }& = & e_\dscr A(O(2)) \otimes \qq \\
{[i_k \qq W^{\otimes_i} , c \qq]^{\mathcal{A}(\dscr)}_* }
& = & \hom_\qq (\qq W^{\otimes_i}, \qq)^W \\
{[c \qq, i_k \qq W^{\otimes_i}  ]^{\mathcal{A}(\dscr)}_* }
& = & \hom_\qq (\qq , \qq W^{\otimes_i})^W \\
{[i_k \qq W^{\otimes_j} , i_k \qq W^{\otimes_i}  ]^{\mathcal{A}(\dscr)}_* }
& = & \hom_\qq (\qq W^{\otimes_i}, \qq W^{\otimes_j})^W \\
{[i_n \qq W^{\otimes_j}  , i_k \qq W^{\otimes_i}  ]^{\mathcal{A}(\dscr)}_* }& = & 0. \\
\end{array}
$$ 
\end{proposition}

\begin{theorem}\label{thm:intrinsicform}
There is an isomorphism of monoidal $dg \qq$-categories
$\h_*\ecal_t \cong \ecal_a$.
\end{theorem}
\begin{proof}
As mentioned above it suffices to give an isomorphism of monoidal
$g \qq$-categories (and hence of $dg \qq$-categories with trivial differentials)
$\pi_* \ecal_{top} \to \ecal_a$. 
We already have a suitable enriched functor: $\underline{\pi}_*$,
our calculations above show that this gives an isomorphism
of enriched categories. This functor respects the monoidal
structures since everything is concentrated in degree zero.
\end{proof}

\begin{theorem}\label{thm:ETtoEA}
Let $C_0$ denote the $(-1)$-connective cover functor on 
$dg \qq$-modules. 
There is a zig-zag of quasi-isomorphisms of monoidal
$dg \qq$-categories.
$$\ecal_t \overset{\sim}{\longleftarrow} C_0 \ecal_t
\overset{\sim}{\longrightarrow} \h_* \ecal_t \cong \ecal_a $$
hence there is a zig-zag of monoidal Quillen equivalences
of $dg \qq \leftmod$-model categories.
$$\rightmod \ecal_t
\overleftarrow{\longrightarrow} \rightmod C_0 \ecal_t
\overrightarrow{\longleftarrow} \rightmod \h_* \ecal_t
\cong \rightmod \ecal_a $$
\end{theorem}
\begin{proof}
This result is \cite[Theorem 4.3.9 and Corollary 4.3.10]{barnes}. 
The map $C_0 \ecal_t \to \ecal_t$ is the inclusion and 
$C_0 \ecal_t \to \h_* \ecal_t$ is the projection. 
These are quasi-isomorphisms because the homology of $\ecal_t$
is concentrated in degree zero. That quasi-isomorphisms induce 
Quillen equivalences of module categories is \cite[Theorem A.1.1]{ss03stabmodcat}.
\end{proof}

\addcontentsline{toc}{part}{Bibliography}
\bibliography{dihedralbib}
\bibliographystyle{alpha}

\end{document}